\documentclass{birkjour}
\usepackage{hyperref}
\usepackage{enumitem}
\usepackage{nth}
\usepackage{mathtools}
\usepackage{tensor}
\usepackage{booktabs}
\usepackage{float}
\usepackage{amsmath, amssymb}
\usepackage{tikz}
\usepackage [english]{babel}
\usepackage [autostyle, english = american]{csquotes}
\MakeOuterQuote{"}
\hypersetup{
    colorlinks=true,
    linkcolor=blue,
    filecolor=blue,
    urlcolor=blue,
    citecolor=blue,
    pdftitle={Overleaf Example},
    pdfpagemode=FullScreen,
    }
\newcommand\restr[2]{{% we make the whole thing an ordinary symbol
  \left.\kern-\nulldelimiterspace % automatically resize the bar with \right
  #1 % the function
  \vphantom{\big|} % pretend it's a little taller at normal size
  \right|_{#2} % this is the delimiter
  }}

 \newtheorem{theorem}{Theorem}[section]
 \newtheorem{Corollary}[theorem]{Corollary}

 \theoremstyle{definition}
 \newtheorem{definition}[theorem]{Definition}
 \theoremstyle{remark}
 \newtheorem{remark}[theorem]{Remark}
 \newtheorem*{example}{Example}
 \numberwithin{equation}{section}

\newcommand*\xoverline[2][0.75]{%
    \sbox{\myboxA}{$\m@th#2$}%
    \setbox\myboxB\null% Phantom box
    \ht\myboxB=\ht\myboxA%
    \dp\myboxB=\dp\myboxA%
    \wd\myboxB=#1\wd\myboxA% Scale phantom
    \sbox\myboxB{$\m@th\overline{\copy\myboxB}$}%  Overlined phantom
    \setlength\mylenA{\the\wd\myboxA}%   calc width diff
    \addtolength\mylenA{-\the\wd\myboxB}%
    \ifdim\wd\myboxB<\wd\myboxA%
       \rlap{\hskip 0.5\mylenA\usebox\myboxB}{\usebox\myboxA}%
    \else
        \hskip -0.5\mylenA\rlap{\usebox\myboxA}{\hskip 0.5\mylenA\usebox\myboxB}%
    \fi}
\makeatother

\begin{document}

%-------------------------------------------------------------------------
% editorial commands: to be inserted by the editorial office
%
%\firstpage{1} \volume{228} \Copyrightyear{2004} \DOI{003-0001}
%\seriesextraline{This is the Concrete Title of this Book\br H.E. R and S.T.C. W, Eds.}
%\firstpage{1}

\title[  $\lambda$-Chromatic Polynomials and Polytope Geometry]
 { $\lambda$-Chromatic Polynomials and Polytope Geometry}

%----------Author 1
\author[A. Ali]{Annayat Ali}

\address{%
National Institute of Technology, Srinagar, \\
Jammu and Kashmir,\\
 190006, India.}

\email{enayatali.57@gmail.com}
%------------Author 2

\author[R. Raja]{Rameez Raja}

\address{%
National Institute of Technology, Srinagar, \\
Jammu and Kashmir,\\
 190006, India.}

\email{rameeznaqash@nitsri.ac.in}



 \date{July 24, 2023}

\begin{abstract}
In this paper, we investigate the notion of the \textit{$\lambda$-chromatic polynomial} of a graph, which enumerates the number of distinct $L(2,1)$-colorings using colors from a prescribed finite set. We prove that the $\lambda$-chromatic polynomial of a graph with $n$ vertices is a monic polynomial of degree $n$ and provide a combinatorial interpretation via lattice point enumeration within the framework of inside-out polytopes. Moreover, we compute the $\lambda$-chromatic polynomial of complete graphs $K_n$ using lattice path enumeration, and we develop a block-gap technique to derive the $\lambda$-chromatic polynomials for complete bipartite and multipartite graphs. Our approach unifies geometric, combinatorial, and algebraic methods to provide a systematic treatment of $\lambda$-colorings across various families of graphs.

\end{abstract}

%%% ----------------------------------------------------------------------
\maketitle
%%% ----------------------------------------------------------------------



\section{Introduction}
Graph coloring represents one of the most extensively studied optimisation challenges in graph theory. For a given simple graph, the objective is to assign a unique colour to each vertex such that no two adjacent vertices share the same colour while minimising the total number of colours used. This problem widely applies to various real-world scenarios, including scheduling, timetabling, frequency assignment, wavelength routing, and many other practical domains (see, \cite{DA, GMO, GKO, QBM}).

In addition to the classical graph colouring problem, numerous generalizations have been proposed to accommodate more complex requirements and constraints in diverse applications (see, for example, \cite{DEW, SDW}). This paper investigates a novel extension of the graph coloring problem, offering greater flexibility in its applications by allowing the selection among multiple predefined strategies for vertex coloring.

Among the many variants of graph colouring, the $L(2,1)$-colouring problem, introduced by Griggs and Yeh \cite{GY}, is particularly important due to its practical relevance in channel assignment problems. The Frequency Assignment Problem (FAP) is a prominent application of the $L(2,1)$-colouring problem. It arises in networks of transmitter/receiver units (TRXs), where each unit must be assigned a frequency while adhering to specific constraints to avoid interference. The first constraint ensures that nearby TRXs do not receive frequencies that are too close, as this may cause signal interference. Interference occurs when two TRXs with identical or nearly identical frequencies are within a certain proximity to a receptor, such as a mobile device. The second constraint stems from the limited availability of the frequency spectrum, making it essential to minimise the total number of frequencies assigned.

This scenario can be modelled using an interference graph, where vertices represent TRXs, and edges exist between vertices if the corresponding TRXs are close enough to interfere. To maintain a compact frequency
bandwidth, their primary concern is to calculate the difference between the
highest and lowest frequencies assigned to the TRXs. The optimal
bandwidth in this context is called the $\lambda$-chromatic number, denoted
by $\lambda$. Hence, the problem can be stated equivalently as colouring the vertices
of a graph with non-negative integers while satisfying the L(2,1)-colouring
constraint.

The concept of the chromatic polynomial was first introduced by George Birkhoff in 1912 during his investigation of the Four Color Problem \cite{BG}. Although this work did not directly resolve the problem of whether every map can be coloured using only four colours such that no two adjacent regions share the same colour, it laid the foundation for a powerful combinatorial tool in graph theory. The chromatic polynomial encodes information about the number of proper colourings of a graph for any given number of colours. Over the years, it has proven to possess fascinating mathematical properties and has been widely studied for its connections to graph structure and enumeration. In this paper, we extend this idea to $L(2,1)$-colouring by introducing the concept of $\lambda$-chromatic polynomials, which count valid $L(2,1)$-colourings of graphs. We explore their behaviour, generalise their computation, and establish results that reveal deeper combinatorial insights. In \cite{ANN}, the authors already developed a frequency allocation algorithm for wireless networks that outperforms classical greedy heuristics by achieving superior computational efficiency and enhanced coverage optimality.

Lattice point enumeration in polytopes is a central topic in discrete geometry and combinatorics, with deep connections to number theory, optimization, and algebraic geometry. At its core, the problem concerns determining how many integer points lie within a given geometric shape, often revealing intricate combinatorial structures. This simple question naturally arises in various mathematical settings, such as counting integer solutions to inequalities, analyzing partition functions, and studying graph colorings. Over time, powerful techniques have been developed to study these counts systematically, leading to rich theories like Ehrhart’s polynomial framework, which describes the growth of lattice points in dilated polytopes. These methods not only provide elegant counting formulas but also offer geometric intuition for understanding complex combinatorial phenomena.

Beyond the combinatorial techniques of graph colouring and lattice path enumeration, geometric methods, particularly lattice point enumeration within polytopes, offer a powerful alternative perspective for analysing colouring problems. By representing valid colourings as integer lattice points inside geometric regions, we can leverage the rich theory of polytopes to model and systematically count the number of admissible colourings. This approach not only unifies combinatorial and geometric viewpoints but also provides new tools to explore colouring constraints through the lens of lattice enumeration and polyhedral geometry.

Lattice path enumeration plays a fundamental role in combinatorial analysis and has significant applications in graph theory. A lattice path consists of a sequence of moves on a grid (lattice), typically constrained to specific directions. By encoding the constraints of $L(2,1)$-labeling into lattice path models, we establish a connection between valid $L(2,1)$-colorings and lattice paths with certain restrictions.

Given the practical relevance of $L(2,1)$-colouring in frequency assignment problems, this paper is dedicated to the study of \textit{$\lambda$-chromatic polynomials}, which enumerate the number of valid $L(2,1)$-colourings of a graph $G$ for a finite colour set. Specifically, we make the following contributions:

\begin{enumerate}
    \item We formally introduce the $\lambda$-chromatic polynomial as a natural extension of the classical chromatic polynomial to the $L(2,1)$-colouring framework.
    \item We prove that the $\lambda$-chromatic function is a polynomial in the size of the colour set, and establish that it is a monic polynomial of degree equal to the order of the graph.
    \item We compute explicit forms of the $\lambda$-chromatic polynomials for complete graphs, complete bipartite graphs, and complete multipartite graphs using combinatorial and geometric methods.
\end{enumerate}

The remainder of the paper is structured as follows. Section~2 provides the essential definitions, background concepts, and relevant prior work that form the foundation for the developments in the later sections. In Section~3, we present the main results of the paper, including the introduction and analysis of the $\lambda$-chromatic polynomial for various classes of graphs.

\section{Preliminary  definitions and  lemmas}
In this section, we introduce the formal framework for $L(2,1)$-colouring of finite graphs and the $\lambda$-chromatic function, along with the necessary background on polytopes, lattice point enumeration, and Ehrhart theory.
\begin{definition}
Let $G$ be a graph. An $L(2,1)$-coloring of a graph $G$ is a function $f:V(G) \longrightarrow \mathbb{Z}_{\geq0}$ that satisfies the following conditions:
\begin{enumerate}[label=(\alph*)]
\item $|f(u)-f(v)| \geq 2$ whenever $u$ is adjacent to $v$;
\item $|f(u)-f(v)| \geq 1$ whenever $u$ is at distance two from $v$.
\end{enumerate}
The \textit{span} of a coloring is defined as the absolute difference between the maximum and minimum values assigned by \( f \). The \textit{\( \lambda \)-chromatic number} of a graph \( G \), denoted by \( \lambda(G) \), is the smallest non-negative integer such that there exists an \( L(2,1) \)-coloring of \( G \) using colors from the set \( \{0, 1, 2, \dots, \lambda(G)\} \). A coloring that achieves this minimum span is called a \textit{\( \lambda \)-coloring} of \( G \).
\end{definition}

\noindent
Note that $\lambda$-colouring of a graph \( G \) is not necessarily unique; specifically, the set of assigned colors may vary. There may exist two distinct \(\lambda \)-colourings that both minimize the span, as illustrated in  Figure \ref{F.111} below. In this example, the two valid image sets are \( \{0,0,2,2,4\} \) and \( \{0,1,2,3,4\} \), represented by red and blue labels, respectively.\\

\begin{figure}[h]
    \centering
    \begin{tikzpicture}[scale=1, every node/.style={inner sep=0pt}, node distance=2cm]
        % Define the five vertices
        \filldraw (-2,0) circle (3pt);  % leftmost
        \filldraw (-1,0) circle (3pt);
        \filldraw (0,0) circle (3pt);   % center
        \filldraw (1,0) circle (3pt);
        \filldraw (2,0) circle (3pt);   % rightmost

        % Draw the edges
        \draw (-2,0) -- (-1,0);
        \draw (-1,0) -- (0,0);
        \draw (0,0) -- (1,0);
        \draw (1,0) -- (2,0);

        % Add red labels (top)
        \node[red, above=5pt] at (-2,0) {0};
        \node[red, above=5pt] at (-1,0) {2};
        \node[red, above=5pt] at (0,0) {4};
        \node[red, above=5pt] at (1,0) {0};
        \node[red, above=5pt] at (2,0) {2};

        % Add blue labels (bottom)
        \node[blue, below=5pt] at (-2,0) {4};
        \node[blue, below=5pt] at (-1,0) {1};
        \node[blue, below=5pt] at (0,0) {3};
        \node[blue, below=5pt] at (1,0) {0};
        \node[blue, below=5pt] at (2,0) {2};
    \end{tikzpicture}
    \caption{Two distinct $\lambda$-colorings of $P_5$, shown with red and blue labels respectively.}
    \label{F.111}
\end{figure}
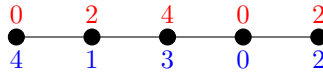

A proper colouring of a graph \( G \) involves assigning colours (positive integers) to the vertices of \( G \) such that no two adjacent vertices share the same colour. The chromatic number of \( G \), denoted by \( \chi(G) \), is the smallest number of colours required for a proper colouring of \( G \).

The chromatic polynomial of a graph, \( P(G, x) \), is a function that expresses the number of ways to properly colour the graph \( G \) using \( x \) colours.

Now, we define the \(\lambda\)-chromatic polynomial for a graph.

\begin{definition}[$\lambda$-Chromatic Polynomial]
For a graph \( G \) and a positive integer \( x \), the \(\lambda\)-chromatic Polynomial, \( \Lambda(G, x): \mathbb{Z}_{\geq0} \longrightarrow \mathbb{N} \), counts the number of possible \( L(2, 1) \)-colorings of the graph \( G \) using colors from the set \( \{0, 1, 2, \dots, x\} \).
\end{definition}

The function \( \Lambda(G, x) \) is well-defined, as for each positive integer \( x \), an \( L(2,1) \)-coloring prescribes a unique assignment of colors to the vertices of \( G \) satisfying two conditions: adjacent vertices receive colors that differ by at least 2, and vertices at distance two differ by at least 1. The number of such valid colorings for any given \( x \) is finite and well-determined for each graph \( G \), ensuring the function is computable in a consistent and meaningful way. The fact that \( \Lambda(G, x) \) is a polynomial will be formally established in the next section.

\noindent

 For any graph \( G \), the problem of finding \( \Lambda(G, x) \) when \( x < \lambda(G) \) is redundant, as in this case, \( \Lambda(G, x) \) is always zero. This is because \( \lambda(G) \) represents the minimum number of colours needed for a valid colouring, and if \( x \) is less than \( \lambda(G) \), it's impossible to have a valid colouring. Therefore, we focus only on the case where \( x \geq \lambda(G) \).
\begin{remark}
Let \( G \) be a graph consisting of \( k \) connected components \( G_1, G_2,$ 
\noindent $\ldots, G_k \). 
The \( \lambda \)-chromatic polynomial of \( G \) satisfies the following factorization property:
\[
\Lambda(G, x) = \prod_{i=1}^k \Lambda(G_i, x).
\]
Since the connected components of \( G \) are disjoint, each component \( G_i \) can be coloured independently using the same set of \( x \) available colours. Thus the total number of valid \( \lambda \)-colourings of \( G \) is obtained by taking the product of the individual colourings of each component.

A direct consequence of this result is the \( \lambda \)-chromatic polynomial of the null graph \( H \) on \( n \) vertices. Since \( H \) has no edges, each vertex can be assigned any of the \( x+1 \) available colours without restriction. This gives:
\[
\Lambda(H, x) = (x+1)^{n}.
\]
\end{remark}

The enumeration of lattice paths is a fundamental problem in combinatorial counting. While lattice paths might appear to be simple mathematical constructs for small-scale problems, their complexity and richness become evident as the scale of the problem increases. On one hand, lattice path enumeration is a powerful technique for analysing discrete structures in fields such as chemistry, physics, probability theory, and computer science. On the other hand, many intriguing properties of lattice paths remain unexplored, making this area a dynamic field of ongoing research.

\vspace{0.3cm}
\noindent
The simplest lattice path problem involves counting the number of paths in the Cartesian plane from the origin \( (0,0) \) to the point \( (m, n) \), using only unit steps to the east and north. The total number of such paths is given by $\binom{m+n}{m}$.

\begin{definition}
Let \( G \) be a graph of order \( n \) with a \( \lambda \)-colouring \( f \). Denote the vertex set of \( G \) as \( V(G) = \{ v_0, v_1, \dots, v_{n-1} \} \), where the vertices are ordered such that \( f(v_i) \leq f(v_j) \) for all \( i < j \), where $0 \leq i, j \leq n-1$. Given a positive integer \( m \geq \lambda(G) \), we define the \textit{Colouring Grid} \( \mathcal{C}(G, f, m) \) as an \( n \times (m+1 - \lambda(G)) \) grid, where the rows correspond to the vertices of \( G \), starting from the bottom. The entry in the cell \( (i,j) \) of \( \mathcal{C}(G, f, m) \) is given by \( f(v_i) + j \).
\end{definition}

\begin{example}
Consider the complete graph \( K_4 \) with the colour set \( \{0, 1, 2, \dots, 10\} \). Its \( \lambda \)-colouring is given by  \( f(v_i)=2i, 0 \leq i \leq 3 \), where \( v_i \in V(K_4) \). The associated colouring grid \( \mathcal{C}(K_4, f, 10) \) is given below.\\

\begin{figure}[h]
\centering
\begin{tikzpicture}[scale=0.8]
% Grid setup
\draw[step=1cm, gray, thin] (0,0) grid (5,4);

% Row Labels (Y-axis)
\node at (-0.5,0.5) {$v_1$};
\node at (-0.5,1.5) {$v_2$};
\node at (-0.5,2.5) {$v_3$};
\node at (-0.5,3.5) {$v_4$};

% Column 1 (0, 2, 4, 6)
\node at (0.5,0.5) {0};
\node at (0.5,1.5) {2};
\node at (0.5,2.5) {4};
\node at (0.5,3.5) {6};

% Column 2 (1, 3, 5, 7)
\node at (1.5,0.5) {1};
\node at (1.5,1.5) {3};
\node at (1.5,2.5) {5};
\node at (1.5,3.5) {7};

% Column 3 (2, 4, 6, 8)
\node at (2.5,0.5) {2};
\node at (2.5,1.5) {4};
\node at (2.5,2.5) {6};
\node at (2.5,3.5) {8};

% Column 4 (3, 5, 7, 9)
\node at (3.5,0.5) {3};
\node at (3.5,1.5) {5};
\node at (3.5,2.5) {7};
\node at (3.5,3.5) {9};

% Column 5 (4, 6, 8, 10)
\node at (4.5,0.5) {4};
\node at (4.5,1.5) {6};
\node at (4.5,2.5) {8};
\node at (4.5,3.5) {10};
\end{tikzpicture}
\vskip 0.2cm
\caption{ \( \mathcal{C}(K_4, f, 10) \).}
\label{fig:coloring-matrix}
\end{figure}
\end{example}
\vspace{-0.5cm}
\noindent
 Each column in the colouring grid of a graph $G$ represents a  $\lambda$-colouring of $G$ with an added constant, while each row corresponding to a vertex $v_i$ lists the possible colours for $v_i$ from the colour set $\{0, 1, 2, \dots, m\}$.\\
 \\
\noindent
The hockey-stick identity is a well-known combinatorial identity given by
\[
\sum_{k=0}^{m} \binom{N+k}{N} = \binom{N+m+1}{N+1}.
\]
This identity is named due to the visual shape it forms in Pascal’s triangle, resembling a hockey stick. It states that the sum of binomial coefficients along a diagonal in Pascal’s triangle results in another binomial coefficient located one row below and one column to the right of the last term in the sum. The identity can be proved using the combinatorial interpretation: the left-hand side counts the number of ways to choose $N$ elements from a set that grows in size from $N$ to $N+m$, while the right-hand side counts the same selection process directly. This identity is particularly useful in counting problems involving lattice paths and combinatorial sums.

We employ the concept of lattice path enumeration to compute the 
\( \lambda \)-chromatic polynomial of the complete graph \( K_n \). We specifically illustrate how 
lattice path techniques can be effectively used in this case. However, we also show in the 
subsequent section that this approach does not extend efficiently to other classes of graphs.

A \textit{convex} polytope is a bounded, non-empty subset of \( \mathbb{R}^d \) that can be expressed as the intersection of finitely many half-spaces, where each half-space is defined by a linear inequality of the form
\[
a_1x_1 + a_2x_2 + \dots + a_dx_d \leq c,
\]
for real constants \( a_1, \dots, a_d \) and \( c \), and where \( (x_1, x_2, \dots, x_d) \in \mathbb{R}^d \). While general polytopes may be non-convex, this combinatorial framework primarily concerns the convex case. Such polytopes are fundamental objects in combinatorial geometry and play a central role in lattice point enumeration. In the classical setting, Ehrhart theory provides a powerful framework for counting the number of lattice points inside dilations of convex polytopes. Specifically, if \( P \) is a rational polytope, the number of lattice points in its \( t \)-th dilation, denoted by \( E_P(t) \), is known to be a quasipolynomial function in \( t \). The leading coefficient of this quasipolynomial corresponds to the normalized volume of the polytope, and its constant term encodes topological information about the polytope.

A hyperplane in \( \mathbb{R}^d \) is a subset of the form \( H = \{ x \in \mathbb{R}^d : a_1x_1 + a_2x_2 + \cdots + a_d x_d = c \} \), where \( a_1, \dots, a_d \in \mathbb{R} \) are not all zero and \( c \in \mathbb{R} \).  A hyperplane arrangement is a finite collection of such hyperplanes, and their union partitions the space (or a convex subset such as a polytope) into finitely many connected components.

The concept of inside-out polytopes, introduced by Beck and Zaslavsky \cite{BZ}, extends Ehrhart’s theory to settings where the polytope is intersected by an arrangement of hyperplanes. An inside-out polytope is formally a pair $(P, \mathcal{H})$, where $P$ is a convex polytope and $\mathcal{H}$ is a finite set of hyperplanes that further dissect the polytope. In this framework, lattice point enumeration involves counting the integer points within $P$ that avoid the hyperplanes in $\mathcal{H}$. 

The authors in \cite{BZ} proved the following interesting result.

\begin{theorem}[Theorem 4.1, \cite{BZ}]
\label{TBZ}
If \( D \) is a full-dimensional discrete lattice and \( (P, \mathcal{H}) \) is a closed, full-dimensional, \( D \)-fractional inside-out polytope in \( \mathbb{R}^d \) such that \( \mathcal{H} \) does not contain the degenerate hyperplane, then \( E_{P, \mathcal{H}}(t) \) and \( E^{\circ}_{P^{\circ}, \mathcal{H}}(t) \) are quasipolynomials in \( t \), with period equal to a divisor of the denominator of \( (P, \mathcal{H}) \), with leading term \( c_d t^d \) where \( c_d = \operatorname{vol}_D P \), and with constant term \( E_{P, \mathcal{H}}(0) \) equal to the number of regions of \( (P, \mathcal{H}) \). Furthermore,
\[
E^{\circ}_{P^{\circ}, \mathcal{H}}(t) = (-1)^d E_{P, \mathcal{H}}(-t).
\]
\end{theorem}

We clarify the definitions of the lattice point enumerators and the associated denominator of an inside-out polytope.

The function \( E_{P, \mathcal{H}}(t) \) denotes the number of lattice points in the dilated polytope \( tP \), each counted with multiplicity equal to the number of closed regions of the inside-out polytope \( (P, \mathcal{H}) \) that contain the point. This count includes points lying on the boundary of the polytope as well as on the hyperplanes in \( \mathcal{H} \). It corresponds to the closed inside-out Ehrhart quasipolynomial.

In contrast, the function \( E^{\circ}_{P^{\circ}, \mathcal{H}}(t) \) counts the number of lattice points in the open polytope \( tP^{\circ} \) that do not lie on any hyperplane in \( \mathcal{H} \). That is, it enumerates the lattice points strictly inside \( tP \) and entirely avoiding the hyperplane arrangement. This is referred to as the open inside-out Ehrhart quasipolynomial.

The denominator of the inside-out polytope \( (P, \mathcal{H}) \), denoted by \( \operatorname{den}(P, \mathcal{H}) \), is the smallest positive integer \( q \) such that all vertices of the dilated polytope \( qP \) and all intersections of hyperplanes in \( \mathcal{H} \) lie in the integer lattice \( \mathbb{Z}^d \). Formally, it is given by
\[
\operatorname{den}(P, \mathcal{H}) = \operatorname{lcm} \left( \operatorname{den}(P), \, \left\{ \operatorname{den}(H) \;\middle|\; H \in \mathcal{H} \right\} \right),
\]
where the denominator of the polytope \( P \) is defined as the smallest positive integer such that the vertices of \( qP \) lie in \( \mathbb{Z}^d \), and the denominator of a rational hyperplane \( H = \{ a_1x_1 + \dots + a_d x_d = b \} \) is the smallest positive integer \( q \) such that \( q a_i \in \mathbb{Z} \) for all \( i \), and \( q b \in \mathbb{Z} \).

According to the theory of inside-out Ehrhart quasipolynomials, the period of both \( E_{P, \mathcal{H}}(t) \) and \( E^{\circ}_{P^{\circ}, \mathcal{H}}(t) \) divides this denominator. In particular, when the denominator is one, that is, when both the polytope and the hyperplanes have integer coefficients, the enumerators are polynomials rather than quasipolynomials.

\section{Main Results}
Given the colouring grid \( \mathcal{C}(K_n, f, m) \) of the complete graph \( K_n \), we label the cells with coordinates \( (x, y) \), where \( 0 \leq x \leq m-2n+2 \) and \( 0 \leq y \leq n-1 \), starting from the bottom-left corner at \( (0, 0) \).

In the following theorem, we establish a bijective correspondence between valid \( L(2,1) \)-colourings of \( K_n \) with maximum colour at most \( m \) and North-East lattice paths in \( \mathcal{C}(K_n, f, m) \). Specifically, given such a colouring, we construct a North-East lattice path starting from a cell in the bottom row (\( y = 0 \)) and ending in the penultimate row (\( y = n-2 \)). Conversely, every North-East lattice path in \( \mathcal{C}(K_n, f, m) \), beginning in the bottom row and terminating in the penultimate row, corresponds to a valid \( L(2,1) \)-colouring of \( K_n \).

\begin{theorem}
\label{t1}
For the complete graph \( K_n \), there is an one-to-one correspondence between valid \( L(2,1) \)-colourings of \( K_n \) with maximum colour at most \( m \) and North-East lattice paths in \( \mathcal{C}(K_n, f, m) \) that begin in the bottom row and terminate in the penultimate row.
\end{theorem}

\begin{proof}
We establish the one-to-one correspondence by demonstrating how to construct a valid \( L(2,1) \)-colouring from a North-East lattice path in the colouring grid and vice versa.

First, consider a North-East lattice path in the colouring grid \( \mathcal{C}(K_n, f, m) \), corresponding to the minimal \( L(2,1) \)-colouring \( f \) of the complete graph \( K_n \) with maximum colour \( m \). This path starts in some cell of the bottom row (\( y = 0 \)) and extends to a cell in the penultimate row (\( y = n-2 \)). To complete the construction, extend this path to the top row by moving one step upward. We now define a new colouring function \( g \) for \( K_n \) by assigning each vertex \( v_i \) the colour of the first cell of the extended North-East lattice path in the \( i \)th row:

\[
g(v_i) =
\begin{aligned}
&\text{colour of the leftmost cell in the } i\text{th row} \\
&\text{of the extended North-East lattice path}.
\end{aligned}
\]

By the structure of the colouring grid, the conditions for an \( L(2,1) \)-colouring are automatically satisfied. Since adjacent vertices correspond to consecutive rows in the lattice path, and each North-East move ensures a column difference of at least 2, adjacent colours differ by at least 2. Similarly, for vertices at distance 2, the colour difference is at least 1, following directly from the grid's structure. Thus, \( g \) defines a valid \( L(2,1) \)-colouring of \( K_n \).

Conversely, suppose we are given an \( L(2,1) \)-colouring \( g \) of \( K_n \). Arrange the assigned colours in increasing order. The corresponding North-East lattice path in \( \mathcal{C}(K_n, f, m) \) is constructed as follows:

Begin at the cell in the bottom row corresponding to the smallest assigned colour, which is located in column \( g(v_0) \). Move upward through the rows, transitioning from the cell corresponding to \( g(v_{i-1}) \) to the cell corresponding to \( g(v_i) \) for each \( 1 \leq i \leq n-2 \). If \( g(v_i) \) is in the same column as \( g(v_{i-1}) \), move directly upward; otherwise, if \( g(v_i) \) is in a strictly greater column, move one step to the right before stepping upward.

Repeating this process constructs a North-East lattice path that extends from the bottom row to the penultimate row.

Thus, we establish an one-to-one correspondence between  North-East lattice paths in the colouring grid and valid \( L(2,1) \)-colourings of \( K_n \).
\end{proof}
\begin{remark}
The forward implication of Theorem~\ref{t1} extends to arbitrary graphs. Specifically, given a colouring grid \( \mathcal{C}(G, f, m) \) associated with a graph \( G \), any North-East lattice path within this grid corresponds to a valid \( L(2,1) \)-colouring of \( G \). This follows directly from the construction outlined in Theorem~\ref{t1}, where the first occurrence of the path in each row determines the colour of the corresponding vertex except for the last row for which we move directly upward.
\end{remark}

The example below illustrates the technique used in Theorem~\ref{t1} and serves as a counterexample demonstrating why the converse does not hold for general graphs.
\begin{example}
Consider the complete bipartite graph \( K_{3,3} \) with partite sets \( V_1 = \{v_0, v_1, v_2\} \) and \( V_2 = \{v_3, v_4, v_5\} \). The \( \lambda \)-colouring function \( f: V(K_{3,3}) \to \mathbb{Z} \) is defined as follows:
\[
f(v_i) =
\begin{cases}
i, & \text{for } 0 \leq i \leq 2, \\
i+1, & \text{for } 3 \leq i \leq 5.
\end{cases}
\]
For \( m = 12 \), the corresponding colouring grid \( \mathcal{C}(K_{3,3}, f, 12) \) along with an extended North-East lattice path is given in Figure~\ref{fig:coloring-K33}
\begin{figure}[h]
\centering
\begin{tikzpicture}[scale=0.8]
\draw[step=1cm, gray, thin] (0,0) grid (7,6);

\node at (-0.5,0.5) {$v_0$};
\node at (-0.5,1.5) {$v_1$};
\node at (-0.5,2.5) {$v_2$};
\node at (-0.5,3.5) {$v_3$};
\node at (-0.5,4.5) {$v_4$};
\node at (-0.5,5.5) {$v_5$};

\node at (0.5,0.5) {0};
\node at (1.5,0.5) {1};
\node at (2.5,0.5) {2};
\node at (3.5,0.5) {3};
\node at (4.5,0.5) {4};
\node at (5.5,0.5) {5};
\node at (6.5,0.5) {6};

\node at (0.5,1.5) {1};
\node at (1.5,1.5) {2};
\node at (2.5,1.5) {3};
\node at (3.5,1.5) {4};
\node at (4.5,1.5) {5};
\node at (5.5,1.5) {6};
\node at (6.5,1.5) {7};

\node at (0.5,2.5) {2};
\node at (1.5,2.5) {3};
\node at (2.5,2.5) {4};
\node at (3.5,2.5) {5};
\node at (4.5,2.5) {6};
\node at (5.5,2.5) {7};
\node at (6.5,2.5) {8};

\node at (0.5,3.5) {4};
\node at (1.5,3.5) {5};
\node at (2.5,3.5) {6};
\node at (3.5,3.5) {7};
\node at (4.5,3.5) {8};
\node at (5.5,3.5) {9};
\node at (6.5,3.5) {10};

\node at (0.5,4.5) {5};
\node at (1.5,4.5) {6};
\node at (2.5,4.5) {7};
\node at (3.5,4.5) {8};
\node at (4.5,4.5) {9};
\node at (5.5,4.5) {10};
\node at (6.5,4.5) {11};

\node at (0.5,5.5) {6};
\node at (1.5,5.5) {7};
\node at (2.5,5.5) {8};
\node at (3.5,5.5) {9};
\node at (4.5,5.5) {10};
\node at (5.5,5.5) {11};
\node at (6.5,5.5) {12};

\draw[line width=1pt, red]
(1.3,0.2) -- (2.3,0.2) -- (3.3,0.2) -- (3.3,1.2) -- (4.3,1.2) -- (4.3,2.2)
(4.3,2.2) -- (5.3,2.2) -- (5.3,3.2) -- (6.3,3.2) -- (6.3,4.2) -- (6.3,5.2);

\fill[black] (1.3,0.2) circle (4pt);
\fill[black] (3.3,1.2) circle (4pt);
\fill[black] (4.3,2.2) circle (4pt);
\fill[black] (5.3,3.2) circle (4pt);
\fill[black] (6.3,4.2) circle (4pt);
\fill[black] (6.3,5.2) circle (4pt);
\end{tikzpicture}
\vskip 0.2cm
\caption{$\mathcal{C}(K_{3,3},f,12)$}
\label{fig:coloring-K33}
\end{figure}

In this grid, the first cell of the North-East lattice path in each row is marked with a dot. The corresponding \( L(2,1) \)-colouring \( g \) derived from this path is given by:
\[
g(v_0) = 1, \quad g(v_1) = 4, \quad g(v_2) = 6, \quad g(v_3) = 9, \quad g(v_4) = 11, \quad g(v_5) = 12.
\]
However, the \( L(2,1) \)-colouring given by
\[
g(v_0) = 2, \quad g(v_1) = 3, \quad g(v_2) = 9, \quad g(v_3) = 5, \quad g(v_4) = 6, \quad g(v_5) = 7
\]
does not correspond to a valid North-East lattice path in the colouring grid \( \mathcal{C}(K_{3,3}, f, 12) \) when arranged in non-decreasing order. This illustrates that not every arbitrary \( L(2,1) \)-colouring of \( K_{3,3} \) necessarily corresponds to a North-East lattice path in the grid.
\end{example}
Below, we determine the \( \lambda \)-chromatic polynomial \( \Lambda(K_n, x) \) for the complete graph \( K_n \).

\begin{theorem}
\label{LKN}
The \( \lambda \)-chromatic polynomial of the complete graph \( K_n \) is given by
\[
\Lambda(K_n, x) = \ (x - n + 2)(x - n + 1)(x-n) \cdots (x - 2n + 3).
\]
\end{theorem}
\begin{proof}
The $\lambda$-colouring of the complete graph \( K_n \) is given by the function
\[
f(v_i) = 2i, \quad \text{where } 0 \leq i \leq n-1,
\]
where the vertex set is \( V(K_n) = \{v_0, v_1, \dots, v_{n-1}\} \). This is the unique $\lambda$-colouring of \( K_n \).

The problem of counting distinct $L(2,1)$-colourings using the colour set \( \{0,1, \dots, m\} \) reduces to counting the number of north-east lattice paths in the colouring grid \( \mathcal{C}(K_n, f, m) \), as suggested by Theorem~\ref{t1}.

More precisely, we count the number of north-east lattice paths from a cell \( (a,0) \) to \( (b,n-2) \), where \( 0 \leq a \leq x - 2n + 2 \) and \( a \leq b \leq x - 2n + 2 \).

Substitute \( N = n - 2 \) and \( M = x - 2n + 2 \). It is well known that the number of north-east lattice paths from \( (0,0) \) to \( (a,b) \) is given by the binomial coefficient:
\[
\binom{a+b}{a}.
\]
First, we count the number of ways to travel from a fixed starting cell \( (i,0) \) to a fixed ending cell \( (j,N) \), where \( 0 \leq i \leq j \leq M \). This count is given by:
\[
\binom{N+(j-i)}{N}.
\]
Summing over all possible paths starting from any cell in the bottom row \( (i,0) \) and ending at any cell in the top row \( (j,N) \), we obtain:
\[
\sum\limits_{i=0}^{M} \sum\limits_{j=i}^{M} \binom{N+(j-i)}{N}.
\]
Now, setting \( k = j - i \), we rewrite the sum as:
\[
\sum\limits_{i=0}^{M} \sum\limits_{k=0}^{M-i} \binom{N+k}{N}.
\]
Applying the hockey-stick identity,
\[
\sum\limits_{k=0}^{m} \binom{N+k}{N} = \binom{N+m+1}{N+1},
\]
we simplify the inner sum to:
\[
\binom{N+(M-i)+1}{N+1}.
\]
Now, summing over all \( i \) from \( 0 \) to \( M \), we again apply the hockey-stick identity:
\[
\sum\limits_{i=0}^{M} \binom{N+(M-i)+1}{N+1} = \binom{N+M+2}{N+2}.
\]
Finally, substituting back \( M = x- 2n + 2 \) and \( N = n - 2 \), we obtain:
\[
\binom{x-n+2}{n}.
\]
Given an $L(2,1)$-colouring of \( K_n \), we can permute the colours, and the colouring remains valid. Since there are \( n! \) ways to permute the colours, we multiply by \( n! \) to account for these permutations. Therefore,

\[
\Lambda(K_n, x) = (x-n+2)(x-n+1)\cdots (x-2n+3).
\]

 This completes the proof.
\end{proof}
\begin{Corollary}
Let \( x, n \in \mathbb{N} \) with \( x \geq 2n - 2 \). Then the number of subsets \( S \subseteq \{0, 1, 2, \dots, x\} \) of size \( |S| = n \) which satisfy the pairwise spacing condition
\[
\min_{s_i,\, s_j \in S,\; i \neq j} |s_i - s_j| \geq 2,
\] is given by
\[
\binom{x - n + 2}{n}.
\]
\end{Corollary}

\begin{proof}
Let \( S \subseteq \{0, 1, 2, \dots, x\} \) be a subset of size \( |S| = n \) such that any two distinct elements of \( S \) differ by at least 2.

Observe that this condition is precisely the defining constraint for an \( L(2,1) \)-coloring of the complete graph \( K_n \).

Conversely, any \( L(2,1) \)-coloring of \( K_n \) using \( n \) distinct labels from \( \{0, 1, \dots, x\} \) that respect the coloring constraint naturally yields such a subset \( S \) as the set of colors used. Since the labels are unordered and distinct, this correspondence is bijective up to permutations of vertex labels.

Therefore, the number of such subsets \( S \) is equal to the number of (unlabeled) \( L(2,1) \)-colorings of \( K_n \) using \( n \) values in \( \{0, 1, \dots, x\} \). The result now follows directly from Theorem~\ref{LKN}.
\end{proof}
The following theorem establishes a formula for enumerating the number of ordered pairs of subsets that satisfy a minimum separation constraint between their elements. This enumeration will play a key role in the subsequent determination of the $\lambda$-chromatic polynomial of the complete bipartite graph.
\begin{theorem}
\label{LKMN}
Let \( x, m, n \in \mathbb{N} \) be fixed positive integers such that \( x \geq m + n \). Then the number of ordered pairs of subsets \( (A, B) \) with
\[
A, B \subseteq \{0, 1, 2, \dots, x\}, \quad |A| = m,\quad |B| = n,\quad \min_{a \in A,\; b \in B} |a - b| \geq 2,
\]
is given by
\[
T = \sum_{k=2}^{m+n} \binom{x - m - n + 2}{k} \sum_{a=1}^{k-1} \binom{k}{a} \binom{m - 1}{a - 1} \binom{n - 1}{k - a - 1}.
\]
\end{theorem}
\begin{proof}
We aim to count the number of ordered pairs of disjoint subsets \( (A, B) \subseteq \{0,1,\dots,x\} \) with \( |A| = m \), \( |B| = n \), such that no element of \( A \) is adjacent to any element of \( B \), that is,
\[
\min_{a \in A,\; b \in B} |a - b| \geq 2.
\]
We represent the problem by assigning a label to each position in the set \( \{0, 1, \dots, x\} \), where each position is marked with the symbol \( \alpha \) to indicate membership in \( A \), the symbol \( \beta \) to indicate membership in \( B \), or the symbol \( \circ \) to denote that it is left empty. The labeling must satisfy three conditions. Exactly \( m \) positions are labeled \( \alpha \), and exactly \( n \) positions are labeled \( \beta \). No \( \alpha \)-labeled position is allowed to be adjacent to a \( \beta \)-labeled one, so patterns such as \( \alpha\beta \) or \( \beta\alpha \) are forbidden. On the other hand, consecutive labels of the same type, such as \( \alpha\alpha \) or \( \beta\beta \), are permitted.

To count such configurations, we use a block-and-gap strategy. Let \( p = (x+1) - (m + n) \) denote the number of empty positions. Suppose the selected positions (those labeled \( \alpha \) or \( \beta \)) form \( k \) contiguous blocks. In order to prevent adjacency between an \( \alpha \)-block and a \( \beta \)-block, there must be at least one empty position between each pair of blocks, which consumes at least \( k - 1 \) of the \( p \) available empty slots. The remaining \( p - (k - 1) \) empty positions can be distributed arbitrarily among the \( k+1 \) available locations: before the first block, between the blocks, and after the last block.

By the classical \emph{stars and bars} argument, the number of ways to distribute \( p - (k - 1) \) indistinguishable empty positions among \( k + 1 \) distinguishable gaps is given by
\[
\binom{(p - (k - 1)) + (k + 1) - 1}{k} = \binom{p + 1}{k}.
\]

Next, within these \( k \) blocks, suppose exactly \( a \) of them are labeled \( \alpha \), and the remaining \( b = k - a \) are labeled \( \beta \). There are \( \binom{k}{a} \) ways to choose which of the \( k \) blocks carry the \( \alpha \) label. The elements of \( A \) must then be partitioned into \( a \) blocks, each of positive length, and similarly for the \( B \) elements into \( b \) blocks. The number of such integer compositions is again given by the stars and bars method: there are \( \binom{m - 1}{a - 1} \) ways to split \( m \) \( \alpha \)-elements into \( a \) positive parts, and \( \binom{n - 1}{b - 1} \) ways to do the same for \( n \) \( \beta \)-elements into \( b \) blocks.

Thus, for fixed \( k \) and \( a \), the total number of such valid labelings is given by
\[
\binom{p+1}{k} \binom{k}{a} \binom{m-1}{a-1} \binom{n-1}{b-1}.
\]

Finally, summing over all valid values of \( k \geq 2 \) (since we require at least one block of each type) and over all \( a \) from \( 1 \) to \( k - 1 \), we obtain the total number of desired configurations as
\[
T = \sum_{k=2}^{m+n} \binom{x - m - n + 2}{k} \sum_{a=1}^{k-1} \binom{k}{a} \binom{m - 1}{a - 1} \binom{n - 1}{k - a - 1}.
\]
\end{proof}
To clarify the definitions of blocks and gaps used in the proof, we present a concrete example illustrating a valid labeling configuration.
\begin{example}
Let \( x = 15 \), \( m = 4 \), and \( n = 3 \). We want to label the positions \( \{0,1,\dots,15\} \) using exactly four \( \alpha \)-labels, three \( \beta \)-labels, and leave the rest empty (denoted by \( \circ \)), such that no \( \alpha \)-labeled position is adjacent to a \( \beta \)-labeled one.

Consider the following valid labeling:

\[
\circ\ \alpha\ \alpha\ \circ\ \circ\ \beta\ \beta\ \beta\ \circ\ \circ\ \alpha\ \alpha\ \circ\ \circ\ \circ\ \circ
\]

In this configuration:
\begin{itemize}
    \item The four \( \alpha \)-elements are grouped into two contiguous blocks.
    \item The three \( \beta \)-elements form a single block.
    \item There are at least one \( \circ \) between any \( \alpha \)-block and \( \beta \)-block, avoiding adjacent \( \alpha\beta \) or \( \beta\alpha \) patterns.
    \item Remaining \( \circ \)'s are distributed freely before, between, and after the blocks.
\end{itemize}
\begin{figure}[H]
\centering
\begin{tikzpicture}[scale=0.8]
% Draw the 16 positions
\foreach \i in {0,...,15} {
    \draw (\i,0) circle (5pt);
    \node at (\i,-0.7) {\tiny \i};
}

% Fill labels
\node at (1,0) {\textcolor{red}{$\alpha$}};
\node at (2,0) {\textcolor{red}{$\alpha$}};
\node at (5,0) {\textcolor{blue}{$\beta$}};
\node at (6,0) {\textcolor{blue}{$\beta$}};
\node at (7,0) {\textcolor{blue}{$\beta$}};
\node at (10,0) {\textcolor{red}{$\alpha$}};
\node at (11,0) {\textcolor{red}{$\alpha$}};

% Add legend
\node at (3.5,1.2) {\scriptsize \textcolor{red}{$\alpha$}-label};
\node at (8.5,1.2) {\scriptsize \textcolor{blue}{$\beta$}-label};
\node at (13.5,1.2) {\scriptsize $\circ$ = empty};
\end{tikzpicture}
\caption{An example of a valid labeling of positions \( 0 \) to \( 15 \), illustrating blocks of \( \alpha \) and \( \beta \) separated by empty positions.}
\label{fig:block-example}
\end{figure}
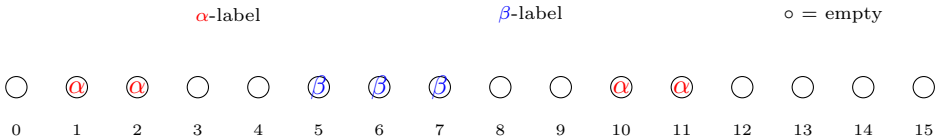\end{example}

\begin{Corollary}
\label{CKMN}
The $\lambda$-chromatic polynomial of the complete bipartite graph $K_{m,n}$ is given by 
\[ 
\Lambda(K_{m,n},x)= m!n!\sum_{k=2}^{m+n} \binom{x - m - n + 2}{k} \sum_{a=1}^{k-1} \binom{k}{a} \binom{m - 1}{a - 1} \binom{n - 1}{k - a - 1}.
\]
\end{Corollary}

\begin{proof}
Let \( x \in \mathbb{N} \), and consider a pair of disjoint subsets \( (A, B) \subseteq \{0, 1, 2, \dots, x\} \) with \( |A| = m \), \( |B| = n \), and satisfying the separation condition
\[
\min_{a \in A,\; b \in B} |a - b| \geq 2.
\]
Such a pair naturally defines an \( L(2,1) \)-coloring of the complete bipartite graph \( K_{m,n} \) by assigning the vertices of the partite set of size \( m \) the values from the set \( A \), and assigning the vertices of the other partite set the values from \( B \). 

Conversely, any \( L(2,1) \)-coloring of \( K_{m,n} \) using values from \( \{0,1,\dots,x\} \) gives rise to such a pair \( (A, B) \), where \( A \) and \( B \) are the sets of colors assigned to the two partite sets, respectively. Since colors can be permuted arbitrarily within each partite set without violating the coloring constraints, the number of such valid colorings (up to these permutations) corresponds exactly to the number of such subset pairs.

Therefore, by Theorem~\ref{LKMN}, the result follows.
\end{proof}

\begin{theorem}
\label{LCBG}
Let \( x \in \mathbb{N} \), and let \( m_1, m_2, \dots, m_n \in \mathbb{N} \) be fixed positive integers with total sum \( M = m_1 + \cdots + m_n \), and assume \( x \ge M + n - 2 \). Then the number of ordered \( n \)-tuples of disjoint subsets
\[
(A_1, A_2, \dots, A_n) \subseteq \{0, 1, \dots, x\}
\]
with \( |A_i| = m_i \) for all \( i \), and such that every element of \( A_i \) is at distance at least 2 from every element of \( A_j \) for \( i \ne j \), is given by
\[
T' = \sum_{k = n}^{M} \binom{x - M + 2}{k}
\sum_{\substack{a_1 + \cdots + a_n = k \\ a_i \ge 1}} 
\binom{k}{a_1, \dots, a_n}
\prod_{i = 1}^{n} \binom{m_i - 1}{a_i - 1}.
\]
\end{theorem}

\begin{proof}
The result follows as a natural generalization of Theorem~\ref{LKMN}, which handles the case \( n = 2 \). We again apply the block-and-gap strategy: each subset \( A_i \) of size \( m_i \) is partitioned into \( a_i \ge 1 \) contiguous blocks, where the total number of blocks is \( k = a_1 + \cdots + a_n \). The requirement that elements from different subsets be at distance at least two is enforced by placing at least one empty position between any two adjacent blocks of distinct types.

As before, the number of empty positions is \( p = x + 1 - M \). Distributing the
\( p - (k-1) \) free empty positions among \( k+1 \) gaps gives
\[
\binom{(p-(k-1))+k}{k} = \binom{p+1}{k} = \binom{x - M + 2}{k}.
\]
Once the total number of blocks \( k \) is fixed, the number of ways to assign block types to the \( n \) subsets is counted by the multinomial coefficient $$\binom{k}{a_1, \dots, a_n}.$$ 
The internal distribution of the \( m_i \) elements of each subset \( A_i \) across its \( a_i \) blocks is given by $$ \binom{m_i - 1}{a_i - 1},$$ using the standard stars-and-bars argument.

Summing over all compositions \( (a_1, \dots, a_n) \) of \( k \) with \( a_i \ge 1 \), and over all \( k \ge n \), yields the stated formula.
\end{proof}

\begin{Corollary}
The \( \lambda \)-chromatic polynomial of the complete \( n \)-partite graph 
\( K_{m_1,m_2,\dots,m_n} \), is given by
\begin{align*}
\Lambda(K_{m_1,m_2,\dots,m_n}, x)
= {} & \left( \prod_{i=1}^n m_i! \right)
\sum_{k = n}^{M} \binom{x - M + 1}{k} \\
& \times \sum_{\substack{a_1 + \cdots + a_n = k \\ a_i \ge 1}} 
\binom{k}{a_1, \dots, a_n}
\prod_{i = 1}^{n} \binom{m_i - 1}{a_i - 1},
\end{align*}
where \( M = m_1 + \cdots + m_n .\)
\end{Corollary}
\begin{proof}
The result is a direct consequence of Theorem~\ref{LCBG}. The proof proceeds in the same manner as Corollary~\ref{CKMN}. 
\end{proof}
The crux of the article lies in the following theorem, which establishes that $\Lambda(G,x)$ is invariably a polynomial function of $x$.
\begin{theorem}
\label{Poly}
Let \( G = (V, E) \) be a finite, undirected graph with \( |V| = n \). Then the \( \lambda \)-chromatic function \( \Lambda(G, x) \) is a monic polynomial in \( x \) of degree \( n \).
\end{theorem}

\begin{proof}
We interpret \( \Lambda(G, x) \) as a lattice point enumerator for an inside-out
polytope \( (P, \mathcal{H}) \) in the sense of Beck and Zaslavsky~\cite{BZ},
and apply Theorem~\ref{TBZ} to conclude polynomiality.
 
Let \( P = [0,1]^n \) be the unit cube in \( \mathbb{R}^n \), with vertices in
\( \mathbb{Z}^n \) and denominator \( \operatorname{den}(P) = 1 \).
A colouring of \( G \) using colours from \( \{0, 1, \dots, x\} \) corresponds
to an integer vector \( \mathbf{c} = (c_1, \dots, c_n) \).
The \( L(2,1) \)-constraints forbid certain differences between coordinates.
Specifically, for each edge \( \{i,j\} \in E \), the values \( |c_i - c_j| \in \{0,1\} \)
are forbidden, giving the hyperplanes
\[
H_{ij}^{(r)} = \bigl\{\mathbf{c} \in \mathbb{R}^n : c_i - c_j = r\bigr\},
\quad r \in \{-1, 0, 1\},
\]
and for each pair \( \{i,k\} \) at distance two in \( G \), the value \( c_i = c_k \)
is forbidden, giving \( H_{ik}^{(0)} = \{c_i = c_k\} \).
The full forbidden arrangement is
\[
\mathcal{H} = \bigl\{H_{ij}^{(r)} : \{i,j\} \in E,\, r \in \{-1,0,1\}\bigr\}
\cup \bigl\{H_{ik}^{(0)} : d_G(i,k) = 2\bigr\}.
\]
Every hyperplane in \( \mathcal{H} \) has integer coefficients and integer
right-hand side, so \( \operatorname{den}(\mathcal{H}) = 1 \), and therefore
\( \operatorname{den}(P,\mathcal{H}) = 1 \).
 
Define the shifted colouring vector \( \mathbf{c}' = \mathbf{c} + \mathbf{1} \),
where \( \mathbf{1} = (1,\dots,1) \). Since each hyperplane in \( \mathcal{H} \)
is defined by a difference \( c_i - c_j = r \), it is invariant under
this shift: \( c_i - c_j = r \) if and only if \( c_i' - c_j' = r \).
The colour set \( \{0, 1, \dots, x\} \) maps to \( \{1, 2, \dots, x+1\} \)
under the shift, so
$$
\Lambda(G, x)
= \#\bigl\{\mathbf{c} \in \{0,\dots,x\}^n : \mathbf{c} \notin H
   \text{ for all } H \in \mathcal{H}\bigr\}$$
$$\hskip 2.2cm = \#\bigl\{\mathbf{c}' \in \{1,\dots,x+1\}^n : \mathbf{c}' \notin H
   \text{ for all } H \in \mathcal{H}\bigr\}.
$$
The set \( \{1, \dots, x+1\}^n \) is precisely the set of integer points
in the open polytope \( (x+2)P^{\circ} = (0, x+2)^n \). Therefore,
\[
\Lambda(G, x) = E^{\circ}_{P^{\circ},\mathcal{H}}(x+2),
\]
where \( E^{\circ}_{P^{\circ},\mathcal{H}}(t) \) denotes the open inside-out
Ehrhart enumerator of \( (P, \mathcal{H}) \).

Equivalently a valid \( L(2,1) \)-colouring corresponds exactly to an integer point in 
\( [0,x+2]^n \) that does not lie on any hyperplane in \( \mathcal{H} \).

Therefore,
\[
\Lambda(G,x) = \#\big( x+2P \cap \mathbb{Z}^n \setminus \bigcup \mathcal{H} \big)
= E^{\circ}_{P,\mathcal{H}}(x+2),
\]
where $x+2P = [0, x+2]^n$ represents the dilation of the unit cube $[0, 1]^n$.

Since \( P = [0,1]^n \) is a closed, full-dimensional, \( \mathbb{Z}^n \)-fractional
inside-out polytope with \( \operatorname{den}(P,\mathcal{H}) = 1 \),
Theorem~\ref{TBZ} implies that \( E^{\circ}_{P^{\circ},\mathcal{H}}(t) \) is a
quasipolynomial in \( t \) with period dividing \( \operatorname{den}(P,\mathcal{H}) = 1 \).
A quasipolynomial of period \( 1 \) is an ordinary polynomial, so
\( E^{\circ}_{P^{\circ},\mathcal{H}}(t) \) is a polynomial in \( t \) of degree
\( n \) with leading coefficient \( \operatorname{vol}_{\mathbb{Z}^n}(P) = 1 \).
 
Since \( \Lambda(G,x) = E^{\circ}_{P^{\circ},\mathcal{H}}(x+2) \) and
\( E^{\circ}_{P^{\circ},\mathcal{H}}(t) \) is a polynomial of degree \( n \) in \( t \)
with leading coefficient \( 1 \), substituting \( t = x+2 \) gives a polynomial
in \( x \) of the same degree \( n \) and the same leading coefficient \( 1 \).

Therefore \( \Lambda(G, x) \) is a monic polynomial in \( x \) of degree \( n \).
\end{proof}
 
\begin{remark}
In the formulation of the \( \lambda \)-chromatic polynomial \( \Lambda(G, x) \), we utilize the open Ehrhart quasipolynomial \( E^{\circ}_{P, \mathcal{H}}(x+2) \). This is because a valid \( L(2,1) \)-coloring must avoid all forbidden configurations, including those where the coloring variables satisfy any of the constraints defining the hyperplanes in \( \mathcal{H} \). The closed count would include lattice points lying on these hyperplanes, thus violating the coloring rules. Therefore, to precisely enumerate the valid colorings, we restrict to the lattice points strictly inside the polytope and away from all hyperplanes, which corresponds exactly to the open inside-out count. In fact, the shift \( \mathbf{c}' = \mathbf{c} + \mathbf{1} \) in the open inside-out enumerator \( E^{\circ}_{P^{\circ},\mathcal{H}}(t) \)
counts lattice points in the \emph{open} polytope \( (0,t)^n = \{1,\dots,t-1\}^n \),
which would exclude the colour \( 0 \) if applied directly at \( t = x+1 \).
The shift maps the colour set \( \{0,\dots,x\} \) to \( \{1,\dots,x+1\} \),
so the dilation parameter is \( t = x+2 \).
This identification has been verified computationally: for \( G = K_n \)
with \( n = 2, 3, 4 \), the values \( E^{\circ}_{P^{\circ},\mathcal{H}}(x+2) \)
match the count of \( L(2,1) \)-colourings exactly.
\end{remark}

Now, we provide following two examples which illustrates Theorem \ref{Poly} on two classes of graphs.

\begin{example}
Consider the path graph \( P_3 \), consisting of three vertices \( v_1, v_2, v_3 \)
and two edges \( \{v_1,v_2\} \) and \( \{v_2,v_3\} \).
 
A colouring of \( P_3 \) using colours from \( \{0,1,\dots,x\} \) is a
vector \( (c_1,c_2,c_3) \in \{0,\dots,x\}^3 \).
The ambient polytope is \( P = [0,1]^3 \).
Its eight vertices are the points of \( \{0,1\}^3 \subset \mathbb{Z}^3 \),
so they already lie in the integer lattice and
\( \operatorname{den}(P) = 1 \).
 
The \( L(2,1) \)-constraints for \( P_3 \) are:
\[
|c_1 - c_2| \geq 2, \qquad |c_2 - c_3| \geq 2
\qquad (\text{edge constraints}),
\]
\[
c_1 \neq c_3 \qquad (\text{distance-two constraint}).
\]
These give the forbidden hyperplane arrangement
\[
\mathcal{H} = \bigl\{
  H_{12}^{-1},\; H_{12}^{0},\; H_{12}^{1},\;
  H_{23}^{-1},\; H_{23}^{0},\; H_{23}^{1},\;
  H_{13}^{0}
\bigr\},
\]
where \( H_{ij}^{(r)} = \{(c_1,c_2,c_3)\in\mathbb{R}^3 : c_i - c_j = r\} \).
Every hyperplane is defined by an equation with integer coefficients and
integer right-hand side, so \( \operatorname{den}(\mathcal{H}) = 1 \),
and therefore
\[
\operatorname{den}(P,\mathcal{H})
= \operatorname{lcm}\!\bigl(\operatorname{den}(P),\,
  \operatorname{den}(\mathcal{H})\bigr)
= \operatorname{lcm}(1,1) = 1.
\]
The valid \( L(2,1) \)-colourings of \( P_3 \) using colours from
\( \{0,1,\dots,x\} \) are the integer points in \( \{0,\dots,x\}^3 \)
that avoid every hyperplane in \( \mathcal{H} \).
Apply the coordinate shift \( c_i' = c_i + 1 \).
Since each hyperplane in \( \mathcal{H} \) involves only a difference
\( c_i - c_j \), it is invariant under this shift:
\( c_i - c_j = r ~\text{if and only if}~ c_i' - c_j' = r \).
The colour set \( \{0,1,\dots,x\} \) maps to \( \{1,2,\dots,x+1\} \),
which is the set of integer points in \( (0,x+2)^3 = (x+2)P^{\circ} \).
Therefore
\[
\Lambda(P_3, x) = E^{\circ}_{P^{\circ},\mathcal{H}}(x+2).
\]
To see explicitly why the shift is necessary, take \( x = 4 \).
The valid colourings of \( P_3 \) with colours from \( \{0,1,2,3,4\} \)
include, for instance, \( (0,2,4) \), \( (0,3,1) \), and \( (4,2,0) \).
In total there are \( 18 \) such colourings, so \( \Lambda(P_3,4) = 18 \).
After shifting, these become \( (1,3,5) \), \( (1,4,2) \), \( (5,3,1) \),
all lying in \( \{1,2,3,4,5\}^3 = (0,6)^3 \cap \mathbb{Z}^3 \),
consistently with \( E^{\circ}_{P^{\circ},\mathcal{H}}(6) = 18 \).
Had we used \( t = x+1 = 5 \) instead, we would obtain
\( E^{\circ}_{P^{\circ},\mathcal{H}}(5) = 4 \neq 18 \), since the
unshifted colour \( 0 \) would have been excluded.
 
Since \( \operatorname{den}(P,\mathcal{H}) = 1 \), Theorem~\ref{TBZ} implies that \( E^{\circ}_{P^{\circ},\mathcal{H}}(t) \)
is a polynomial in \( t \) of degree \( 3 \) with leading coefficient
\( \operatorname{vol}([0,1]^3) = 1 \). Substituting \( t = x+2 \), we
conclude that \( \Lambda(P_3, x) \) is a monic polynomial in \( x \) of
degree \( 3 \).
 
By direct enumeration or by the lattice point count,
\[
\Lambda(P_3, x) = x^3 - 4x^2 + 5x - 2,
\]
which is indeed monic of degree \( 3 = |V(P_3)| \), confirming the Theorem \ref{Poly}.
The values for small \( x \) are recorded in Table~\ref{tab:P3}.
 
\begin{table}[H]
\centering
\caption{Values of \( \Lambda(P_3, x) = x^3 - 4x^2 + 5x - 2 \).}
\label{tab:P3}
\renewcommand{\arraystretch}{1.2}
\begin{tabular}{ccccccc}
\toprule
\( x: \) & \( 2 \) & \( 3 \) & \( 4 \) & \( 5 \) & \( 6 \) & \( 7 \) \\
\midrule
\( \Lambda(P_3,x): \) & \( 0 \) & \( 4 \) & \( 18 \) & \( 48 \) & \( 100 \) & \( 180 \)\\ 
\bottomrule
\end{tabular}
\end{table}
\vskip -2.7cm
Note that \( \Lambda(P_3,2) = 0 \), consistent with the fact that the
\( \lambda \)-number of \( P_3 \) is \( 3 \): no valid colouring exists
when the colour set has fewer than \( 4 \) elements.
The third finite differences of the sequence \( 0, 4, 18, 48, 100, 180,\dots \)
are all equal to \( 6 = 3! \), confirming that the leading coefficient is
\( 1 \), as required.
\end{example}

\begin{example}
Consider the complete graph \( K_3 \) with vertex set \( V = \{v_1, v_2, v_3\} \).
The colour space for all possible \( L(2,1) \)-colourings of the three vertices
using colours from \( \{0, 1, \dots, x\} \) is represented by the cube
\( \{0,\dots,x\}^3 \subset \mathbb{R}^3 \), corresponding to the dilated
unit cube \( (x+2)P \) under the coordinate shift described in Theorem \ref{Poly}.
 
For each edge in \( K_3 \), the \( L(2,1) \)-colouring constraints introduce the
following forbidden hyperplanes in \( \mathbb{R}^3 \):
\begin{itemize}
    \item For each edge \( \{i,j\} \), the hyperplane
          \( H_{(i,j)}^0 = \{ x_i = x_j \} \) prevents adjacent vertices
          from receiving the same colour.
    \item The hyperplanes \( H_{(i,j)}^1 = \{ x_i - x_j = 1 \} \) and
          \( H_{(j,i)}^1 = \{ x_j - x_i = 1 \} \) ensure that adjacent
          vertices do not receive colours differing by exactly one.
\end{itemize}
Since \( K_3 \) is a complete graph, every pair of vertices is at distance one,
and there are no pairs at distance two. Accordingly, no additional hyperplanes
of the form \( \{x_i = x_k\} \) are needed for distance-two constraints.
The full forbidden arrangement \( \mathcal{H} \) therefore consists of nine
hyperplanes: three edges, each contributing three hyperplanes
\( H_{(i,j)}^0 \), \( H_{(i,j)}^1 \), and \( H_{(j,i)}^1 \).
 
The inside-out polytope is visualised in Figure~\ref{Figure_2}.
The cube represents the allowable colour space, while the arrangement of
forbidden hyperplanes partitions it into admissible and inadmissible regions.
 
\begin{figure}[H]
    \centering
    \includegraphics[width=.8\textwidth]{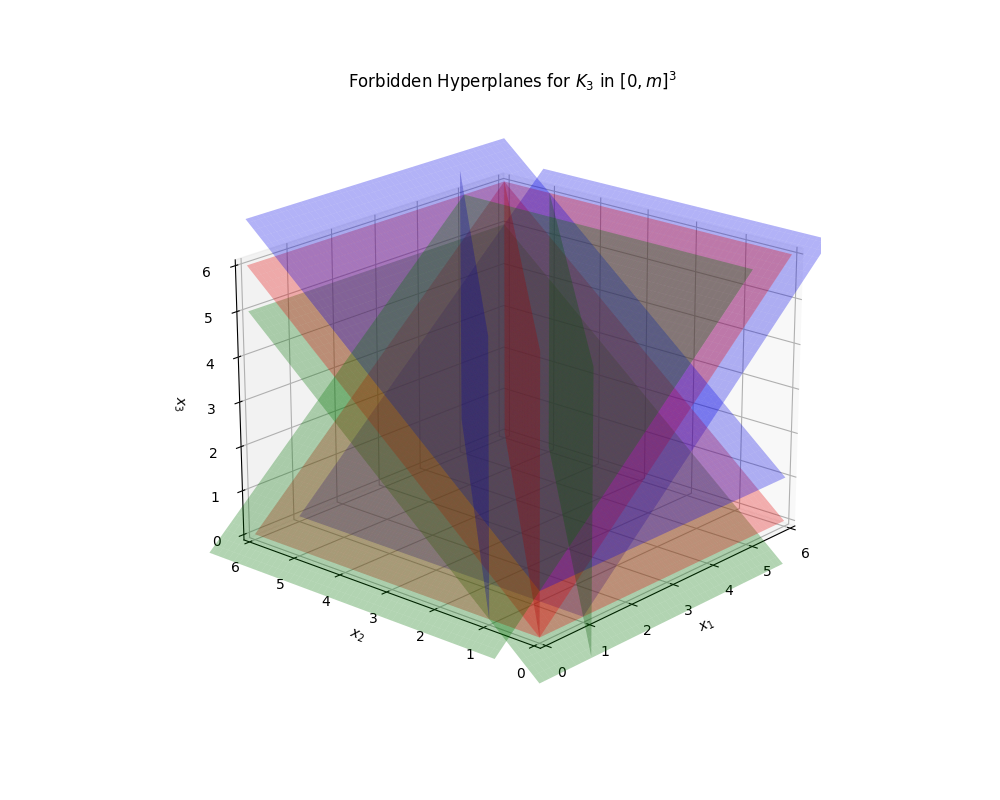}
    \caption{Inside-out polytope for \( K_3 \) showing the forbidden hyperplanes
             corresponding to \( \lambda \)-colouring constraints.}
    \label{Figure_2}
\end{figure}

This diagram illustrates how the hyperplanes slice the polytope into multiple regions and how the admissible integer points (corresponding to proper $\lambda$-colorings) are located strictly within regions that avoid the hyperplanes.
\end{example}
\begin{remark}
In classical combinatorics, Stanley~\cite{S} showed that the chromatic
polynomial \( P(G,x) \) satisfies the reciprocity formula
\[
P(G,-1) = (-1)^n \cdot (\text{number of acyclic orientations of } G),
\]
where \( n = |V(G)| \). For example, \( P(K_3, x) = x(x-1)(x-2) \),
and \( P(K_3,-1) = (-1)(-2)(-3) = -6 = (-1)^3 \cdot 6 \), where \( 6 = 3! \)
is the number of acyclic orientations of \( K_3 \).
 
In the setting of \( L(2,1) \)-colouring, evaluating \( \Lambda(G,x) \) at
\( x = -1 \) yields a different, and more subtle, result.
From the identification \( \Lambda(G,x) = E^{\circ}_{P^{\circ},\mathcal{H}}(x+2) \)
and the Beck--Zaslavsky reciprocity theorem \cite{BZ}, which states
\( E^{\circ}_{P^{\circ},\mathcal{H}}(t) = (-1)^n E_{P,\mathcal{H}}(-t) \),
we obtain
\[
\Lambda(G,x) = (-1)^n E_{P,\mathcal{H}}(-x-2).
\]
Evaluating at \( x = -1 \) gives
\[
\Lambda(G,-1) = (-1)^n E_{P,\mathcal{H}}(-1),
\]
where \( E_{P,\mathcal{H}}(t) \) is the closed inside-out Ehrhart polynomial
of the pair \( (P,\mathcal{H}) \). By Beck and Zaslavsky's theorem, the region count is encoded
in the constant term \( E_{P,\mathcal{H}}(0) \), which corresponds to
\( x = -2 \) rather than \( x = -1 \).
 
For the graph \( K_3 \) and the unit cube \( P = [0,1]^3 \),
as illustrated in Figure~\ref{Figure_2},
the forbidden arrangement \( \mathcal{H} \) for \( K_3 \) consists of nine
hyperplanes: for each of the three edges \( \{i,j\} \), the hyperplanes
\( H_{ij}^{(r)} = \{c_i - c_j = r\} \) for \( r \in \{-1, 0, 1\} \).
Of these nine hyperplanes, only the three given by \( r = 0 \), namely
\[
\{c_1 = c_2\}, \quad \{c_1 = c_3\}, \quad \{c_2 = c_3\},
\]
intersect the interior of \( [0,1]^3 \). The remaining six hyperplanes,
defined by \( c_i - c_j = \pm 1 \), satisfy \( |c_i - c_j| = 1 \) only
when one coordinate equals \( 0 \) and the other equals \( 1 \); this
forces both points to lie on the boundary of \( [0,1]^3 \), so
these six hyperplanes do not cut the interior of the cube. This is
illustrated in Figure~\ref{fig:hyperplane-k3}.
 
The three effective hyperplanes \( \{c_i = c_j\} \) partition the interior
of \( [0,1]^3 \) into the six open regions corresponding to the six strict
orderings of \( (c_1, c_2, c_3) \): $c_1 < c_2 < c_3, \quad c_1 < c_3 < c_2, \quad c_2 < c_1 < c_3,
\quad c_2 < c_3 < c_1, \quad c_3 < c_1 < c_2, \quad c_3 < c_2 < c_1$.

Therefore the arrangement \( \{c_1 = c_2,\, c_1 = c_3,\, c_2 = c_3\} \)
produces exactly \(6 \) open regions inside \( [0,1]^3 \),
as illustrated in Figure~\ref{fig:hyperplane-k3}.
 
\begin{figure}[H]
    \centering
    \includegraphics[width=0.6\textwidth]{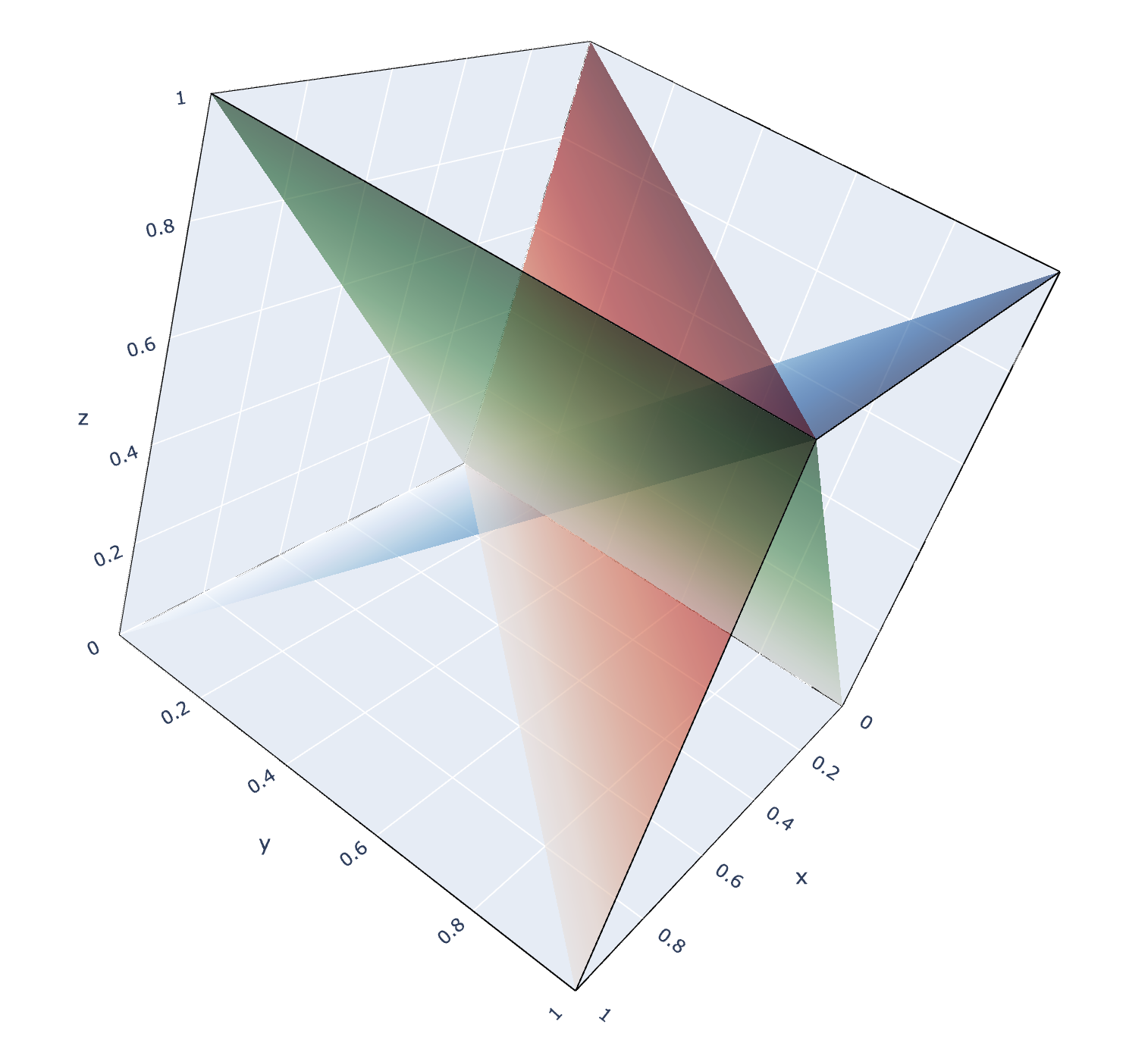}
    \caption{Hyperplane arrangement in \( [0,1]^3 \) for \( K_3 \), showing the
             three effective hyperplanes \( \{c_i = c_j\} \) and the resulting
             partition into \(6 \) open regions, each corresponding to
             a strict ordering of \( (c_1, c_2, c_3) \).}
    \label{fig:hyperplane-k3}
\end{figure}

From Theorem~\ref{LKN}, the \( \lambda \)-chromatic polynomial of \( K_3 \) is
\( \Lambda(K_3, x) = (x-1)(x-2)(x-3) \).
Evaluating at \( x = -1 \):
\[
\Lambda(K_3,-1) = (-2)(-3)(-4) = -24 = (-1)^3 \cdot 24.
\]
More generally, for \( K_n \) one computes
\[
\Lambda(K_n,-1) = (-1)^n \cdot \frac{(2n-2)!}{(n-2)!}.
\]
Unlike the classical case, where \( (-1)^n P(G,-1) \) counts acyclic
orientations, the quantity \( (-1)^n \Lambda(K_n,-1) = (2n-2)!/(n-2)! \)
does not admit an obvious analogous combinatorial interpretation in terms of
orientations or regions of the arrangement, and identifying one is an
interesting open problem.
\end{remark}

An $L(h,k)$-coloring of a graph \( G = (V, E) \), where \( h \) and \( k \) are non-negative integers, is an assignment of non-negative integers (colors) to the vertices of \( G \) such that adjacent vertices receive colors that differ by at least \( h \), and vertices at distance two receive colors that differ by at least \( k \). The classical $L(2,1)$-coloring corresponds to the special case when \( h = 2 \) and \( k = 1 \).

We define the \( L(h,k) \)-chromatic function \( \Lambda_{h,k}(G, x) \) to be the number of distinct \( L(h,k) \)-colorings of \( G \) using colors from the set \( \{0, 1, \dots, x\} \).

The inside-out polytope model naturally extends to the $L(h,k)$-coloring by generalizing the forbidden hyperplanes to encode the distance constraints: for adjacent vertices \( u, v \), the hyperplanes \( x_u - x_v = t \) and \( x_v - x_u = t \) for all \( t \in \{0, 1, 2, \dots, h-1\} \) are forbidden, and for vertices at distance two, the hyperplanes \( x_u - x_v = t \) and \( x_v - x_u = t \) for all \( t \in \{0, 1, 2, \dots, k-1\} \) are excluded.

\begin{remark}
\label{rem:lhk}
The argument discussed in Theorem \ref{Poly} applies to the \( L(h,k) \)-chromatic function
\( \Lambda_{h,k}(G,x) \) for any non-negative integers \( h \) and \( k \).
The forbidden hyperplanes become
\( H_{ij}^{(r)} = \{c_i - c_j = r\} \) for \( r \in \{-(h-1),\dots,h-1\} \)
at adjacent pairs, and
\( H_{ik}^{(r)} = \{c_i - c_k = r\} \) for \( r \in \{-(k-1),\dots,k-1\} \)
at distance-two pairs.
All hyperplanes still have integer coefficients, so
\( \operatorname{den}(P,\mathcal{H}) = 1 \) and Theorem~\ref{TBZ} applies
unchanged. The identification \( \Lambda_{h,k}(G,x) = E^{\circ}_{P^{\circ},\mathcal{H}}(x+2) \)
holds by the same shift argument, and \( \Lambda_{h,k}(G,x) \) is therefore a
monic polynomial in \( x \) of degree \( n \).
\end{remark}

\begin{Corollary}
For any finite graph \( G \), the \( L(h,k) \)-chromatic function \( \Lambda_{h,k}(G, x) \) is a monic polynomial in \( x \) of degree \( n \).
\end{Corollary}
\begin{proof}
The proof follows the same structure as the proof of the polynomiality of \( \Lambda(G, x) \) in the \( L(2,1) \)-case. The difference lies in the set of forbidden hyperplanes, which now encode the generalized distance constraints corresponding to \( h \) and \( k \). Since the polytope and all hyperplanes involved have integer coefficients, the counting function remains a polynomial by Beck and Zaslavsky's theorem.
\end{proof}

\section*{Conclusion}
In this article, we investigated the concept of the $\lambda$-chromatic polynomial in a graph, a natural extension of the classical chromatic polynomial tailored to the $L(2,1)$-colouring framework. By establishing a connection with lattice point enumeration in inside-out polytopes, we demonstrated that the $\lambda$-chromatic function is a polynomial whose degree equals the order of the graph and whose leading coefficient is one. The combinatorial structure of this polynomial was explored using both Ehrhart's theory and lattice path enumeration. Furthermore, we derived explicit expressions for the $\lambda$-chromatic polynomials of key graph families such as complete graphs, complete bipartite graphs, and complete multipartite graphs. The inside-out polytope framework provided a rigorous method for systematically excluding forbidden colourings, and the Ehrhart's reciprocity theorem suggested deep combinatorial dualities that warrant further investigation. This study opens several avenues for future work, including the exploration of $\lambda$-chromatic reciprocity in relation to graph orientations and potential extensions to more general $L(h,k)$-colourings. \\
\paragraph{Future Directions.}
A promising avenue for future research is the development of a recursive formula for the $\lambda$-chromatic polynomial, analogous to the classical deletion-contraction relation for chromatic polynomials. Such a recursive structure could provide efficient computational techniques and deeper combinatorial insights into $\lambda$-colourings. Additionally, exploring the combinatorial interpretation of evaluating the $\lambda$-chromatic polynomial at negative integers may reveal further structural properties of graphs under distance-based colouring constraints.

\section*{Declarations}

\subsection*{Competing interests}
The authors declare that there are no financial interests or personal relationships that could have appeared to influence the work discussed in the manuscript.

\subsection*{Funding}
The first author's research is supported by the University Grants Commission, Govt. of India under UGC-Ref. No. 191620023547, and the second author's research work is funded by the Department of Atomic Energy, Govt. of India under S.No. [02011/15/2023NBHM(R.P)/R\&D II/5866].

\subsection*{Availability of data and materials}
Data sharing is not applicable to this article as no datasets were generated or analysed during the current study.

\end{document}